\documentclass[a4paper,11pt]{amsart}

\usepackage{amsmath}
\usepackage{mathtools}
\usepackage[all]{xy}
\usepackage[latin1]{inputenc}
\usepackage{varioref}
\usepackage{amsfonts}
\usepackage{color}
\usepackage{amssymb}
\usepackage{bbm}
\usepackage{graphicx}
\usepackage{mathrsfs}
\usepackage[hypertexnames=false,backref=page,pdftex,
 	pdfpagemode=UseNone,
 	breaklinks=true,
 	extension=pdf,
 	colorlinks=true,
 	linkcolor=blue,
 	citecolor=red,
 	urlcolor=blue,
 ]{hyperref}

\newcommand{\sE}{{\mathcal E}}
\newcommand{\sF}{{\mathcal F}}

\newcommand{\sH}{{\mathcal H}}

\newcommand{\sO}{{\mathcal O}}

\newcommand{\sX}{{\mathcal X}}

\newcommand{\gothD}{{\mathfrak D}}

\newcommand{\gothF}{{\mathfrak F}}

\newcommand{\gothM}{{\mathfrak M}}

\newcommand{\gothX}{{\mathfrak X}}

\newcommand{\gothm}{{\mathfrak m}}

\newcommand{\scrH}{{\mathscr H}}

\newcommand{\scrX}{{\mathscr X}}

\newcommand{\C}{{\mathbb C}}

\newcommand{\N}{{\mathbb N}}
\renewcommand{\P}{{\mathbb P}}
\newcommand{\Q}{{\mathbb Q}}

\newcommand{\Z}{{\mathbb Z}}

\newcommand{\codim}{\operatorname{codim}}

\newcommand{\Hilb}{{\rm Hilb}}

\newcommand{\Hdg}{\operatorname{Hdg}}

\newcommand{\img}{\operatorname{im}}
\newcommand{\into}{{\, \hookrightarrow\,}}
\newcommand{\isom}{{\ \cong\ }}

\newcommand{\NS}{\operatorname{NS}}

\newcommand{\rank}{{\rm rank}}

\newcommand{\ratl}{\dashrightarrow}

\newcommand{\reg}{{\operatorname{reg}}}

\newcommand{\rk}{{\rm rk}}

\newcommand{\Spec}{\operatorname{Spec}}

\newcommand{\Supp}{\operatorname{Supp}}

\renewcommand{\to}[1][]{\xrightarrow{\ #1\ }}

\newcommand{\tensor}{\otimes}

\newcommand{\vphi}{\varphi}
\newcommand{\vrho}{{\varrho}}

\newtheoremstyle{citing}
  {}
  {}
  {\itshape}
  {}
  {\bfseries}
  {\textbf{.}}
  {.5em}
  {\thmnote{#3}}

\theoremstyle{plain}

\newtheorem{thm}[subsection]{Theorem}

\theoremstyle{definition}

\newtheorem{corollary}[subsection]{Corollary}
\newtheorem{definition}[subsection]{Definition}
\newtheorem{Ex}[subsection]{Example}
\newtheorem{lemma}[subsection]{Lemma}
\newtheorem{proposition}[subsection]{Proposition}
\newtheorem{prop}[subsection]{Proposition}
\numberwithin{equation}{section}

\theoremstyle{remark}
\newtheorem{remark}[subsection]{Remark}

{\theoremstyle{citing}
}

\newcommand{\CH}{\operatorname{CH}}

\newcommand{\Def}{\operatorname{Def}}
\newcommand{\txf}{{T_X\langle F\rangle }}
\newcommand{\gothtxf}{{T_{\gothX/S}\langle \gothF\rangle }}
\newcommand{\deff}{{\Def(F,X)}}
\newcommand{\defx}{{\Def(X)}}
\newcommand{\dR}{{\operatorname{dR}}}
\newcommand{\Gr}{{\operatorname{Gr}}}

\title[Coisotropic fibrations on  holomorphic symplectic manifolds]{Stability  of  coisotropic fibrations on  holomorphic symplectic manifolds}

\author{Christian Lehn}
\address{Christian Lehn\\Institut f\"ur Algebraische Geometrie\\Gottfried Wilhelm Leibniz Universit\"at Hannover\\
Welfengarten 1\\30167 Hannover\\Germany}
\email{lehn@math.uni-hannover.de}

\author{Gianluca Pacienza}
\address{Gianluca Pacienza\\Institut de Recherche Math\'ematique
Avanc\'ee\\ Universit\'e de Strasbourg et CNRS\\
7 rue Ren\'e Descartes\\67084 Strasbourg Cedex\\France}
\email{pacienza@math.unistra.fr}

\let\origmaketitle\maketitle
\def\maketitle{
  \begingroup
  \def\uppercasenonmath##1{} 
  \let\MakeUppercase\relax 
  \origmaketitle
  \endgroup
}

\begin{document}
\thispagestyle{empty}

\maketitle

\begin{abstract}
We investigate the stability of  fibers of  coisotropic fibrations on  holomorphic symplectic manifolds and generalize Voisin's result on Lagrangian subvarieties to this framework. We present applications to holomorphic symplectic manifolds which are deformation equivalent  to Hilbert schemes of points on a $K3$ surface or to generalized Kummer manifolds.
\end{abstract}

\setlength{\parindent}{0em}
\setcounter{tocdepth}{1}

\section{Introduction}\label{sec intro}
\thispagestyle{empty}
Let $X$ be a holomorphic symplectic manifold, by which we mean a compact, simply connected K\"ahler manifold  
 such  that $H^0(X,\Omega^2_X)=\langle \sigma_X\rangle$ where $\sigma_X$ is a non-degenerate 2-form. Voisin showed in \cite{Voi89} that if $F\subset X$ is a Lagrangian submanifold, that is, an analytic connected isotropic submanifold (i.e. $(\sigma_X)_{|F}=0$) of maximal dimension $\dim(X)/2$, then the only obstruction to deforming $F$ along with the ambient holomorphic symplectic manifold $X$ is Hodge theoretic.  More precisely, 
 she proved that 
$$
\Hdg_F = \Def(F,X)
$$ 
where $\Hdg_F$ is the closed analytic subset of the Kuranishi space $\Def(X)$ of $X$ where the cohomology class $[F]$ remains of type $(n,n)$, $2n=\dim (X)$
and $\Def(F,X)\subset \Def(X)$ parametrizes deformations of $X$ containing a deformation of $F$. Following Voisin, we call such a property \emph{stability of the submanifold $F$}. 

Recently, Voisin \cite{Voi15} brought to light the importance of a natural generalization of Lagrangian subvarieties in the study of projective holomorphic symplectic manifolds $X$ and their Chow groups of 0-cycles $\CH_0(X)$. To be more precise, we need to recall some definitions. For any point $x\in X$ she considers its {\it rational orbit} $O_x$, i.e.,  the set of points in $X$ which are rationally equivalent to $x$, and for any $k=1,\ldots, n$ she defines $S_k (X)$ (respectively $S_{k} \CH_0(X)$)
to be the subset of points $x\in X$ such that 
$\dim O_x\geq k$ (respectively the subgroup of $\CH_0(X)$ generated by classes of points $x\in S_k(X)$).
She conjectures that this new decreasing filtration $S_{\bullet} \CH_0(X)$ on  $\CH_0(X)$ is opposite to the conjectural Bloch-Beilinson filtration $F_{BB}^\bullet \CH_0(X)$ (in the sense that 
the natural map 
\begin{equation}\label{eq:BB}
S_{k} \CH_0(X)\to \CH_0(X)/F_{BB}^{2(n-k+1)}\CH_0(X)
\end{equation} is an isomorphism) and that it provides the splitting of the Bloch-Beilinson filtration conjectured by Beauville in \cite{B07}, see \cite{Voi15} for the details. 
 
 A pivotal r\^ole in her approach is played by the existence of algebraically coisotropic subvarieties with constant cycle isotropic fibers. Indeed, on the one hand she proves in \cite[Theorem 0.7]{Voi15} that if $P$ is a  subvariety of $S_k(X)$ of maximal dimension $2n-k$, then 
 $P$ is endowed with a dominant rational map $P\dashrightarrow B$ whose fibers $F$ are $k$-dimensional and isotropic (which is the definition of algebraically coisotropic subvarieties) and such that any two points of $F$ are rationally equivalent in $X$. Moreover, she observes that if such subvarieties exist 
 (which she conjectures to be true cf. \cite[Conjecture 0.4]{Voi15}), the axioms of the Bloch-Beilinson filtration  would already imply the surjectivity of the map (\ref{eq:BB}), see \cite[Lemma 3.9]{Voi15}.
 
 The study of the stability of algebraically coisotropic subvarieties seems therefore to be a relevant task in this new and promising research direction. 

Notice that by work of Amerik and Campana \cite[Theorem 1.2]{AC} we know that  a codimension one submanifold of a  projective holomorphic symplectic manifold (of dimension $\geq 4$) is coisotropic if and only if it is uniruled. 
As it became clear in \cite{CP}, where the stability of rational curves covering a divisor has been verified, instead of working with the coisotropic subvariety itself it seems more efficient to deal with the isotropic fibers $F$ and their stability (see \cite[Remark 3.5]{CP} and Remark \ref{rmk:defP} below). This apparently weaker property allows nevertheless to recover the stability of a union of uniruled codimension one subvarieties containing the initial uniruled divisor, see \cite[Corollary 3.4]{CP}. This is the viewpoint that we adopt here.

Our main result confirms the stability 
of fibers of smoothly algebraically coisotropic subvarieties (cf. Definition \ref{def:main}) of any possible dimension and is summarized in the following.

\begin{thm}\label{thm:stability-intro}
Let $X$ be a  compact holomorphic symplectic manifold of dimension $2n$ and let $F\subset X$ be an isotropic submanifold of dimension $k=1,\ldots,n$. 
Suppose that $F$ is the general fiber of a smoothly 
algebraically coisotropic subvariety and that 
\begin{equation}\label{eq:hypo}
(n-k) h^1(\mathcal O_F)=0.
\end{equation}
Then
$$
\Def (F,X)=\Hdg_F
$$
and for any $t\in \Def (F,X)$ the deformation $X_t$ of $X$ contains a codimension $k$ algebraically coisotropic subvariety $P_t$ covered by the deformations $F_t$ of $F$. 
Moreover,
$$
\codim_{\Def(X)} \Def (F,X) = \rk\left(H^2(X,\C) \to H^2(F,\C)\right).
$$
In particular, if $\NS(F)$ has rank $1$, the subvariety $F$ deforms in a hyperplane in the Kuranishi space of $X$.
\end{thm}
As in \cite{Voi89} we restrict ourselves to deforming smooth submanifolds $F$, although it seems plausible that some extensions to the singular case are possible in the spirit of \cite{L1}. Notice however, that the coisotropic subvariety $P\subset X$ swept out by deformations of $F$ inside $X$ is \emph{not} supposed to be smooth.
We do not know whether the theorem holds without the  vanishing hypothesis  (\ref{eq:hypo}), which is automatically satisfied in the Lagrangian case. 
Nevertheless, notice  that the hypothesis (\ref{eq:hypo})
is satisfied 
if $F$ has trivial Chow group $\CH_0(F)_\Q=\Q$ (e.g. rationally connected), a case which is of particular interest for the cycle-theoretic applications envisioned in \cite{Voi15}.

On the other hand, exactly as in \cite{Voi89} one can show (cf. Propositions \ref{proposition voisin1} and \ref{proposition voisin2}) without any cohomological hypothesis  that the Hodge locus of an isotropic submanifold $F\subset X$ coincides with the locus where  $F$ or its cohomology class remain isotropic (i.e. $(\sigma_t)_{|F}=0$ resp. $\sigma_t\cup [F]=0$). As a consequence, we have the following two other characterizations of $\Def (F,X)$.
\begin{corollary}\label{cor:char}
Let $X$ and $F$ be as in Theorem \ref{thm:stability-intro}. 
Then the following are equivalent:
\begin{enumerate}
\item[(i)] $t\in \Def (F,X)$;
\item[(ii)] $\sigma_t\cup[F]=0 \in H^{2k+2}(X_t,\C)\cong H^{2k+2}(X,\C)$, for $0\not =\sigma_t\in H^0(X_t, \Omega^2_{X_t})$;
\item[(iii)] $(j_t)^*(\sigma_t)=0$ where $(j_t)^*$ denotes the composition of the inclusion $j:F\hookrightarrow X$
with the isomorphism $H^2(X_t,\C)\cong H^2(X,\C)$.
\end{enumerate}
\end{corollary}

Note that we always work with germs of spaces, see the introduction to section \ref{sec defo}.
As in \cite[Corollaire 1.4]{Voi89} we also deduce the following. 
\begin{corollary}\label{cor:voisin}
With the notation above we have that $\deff$ coincides with the locus of 
deformations of $X$ preserving the subspace $L^\Q_F$ of $NS(X)_\Q$ defined as the intersection with $H^2(X,\Q)$ of the orthogonal, with respect to the Beauville-Bogomolov quadratic form, to $\ker(\mu_F),$ where $\mu_F$
is the cup-product map
$$
 \mu_F:H^2(X,\C)\to H^{2k+2}(X,\C),\ \eta\mapsto \eta\cup[F].
$$
 \end{corollary}

Turning to holomorphic symplectic manifolds of given deformation type and to the Chow groups of such manifolds,
as an application of Theorem \ref{thm:stability-intro} we obtain the following. 
\begin{corollary}\label{cor:moduli}
Let $\mathfrak M$ (respectively $\mathfrak M^{pol}$) be the moduli space of marked (respectively marked and polarized) holomorphic symplectic manifolds of dimension $2n$ which are deformation equivalent  to Hilbert schemes of points on a $K3$ surface or to generalized Kummer manifolds. Then for any  $k=1,\ldots,n$  there exists a divisor $\mathfrak D_k\subset \mathfrak M$  (respectively $\mathfrak D_k^{pol}\subset \mathfrak M^{pol}$) such that for any $t\in \mathfrak D_k$ (respectively $t\in\mathfrak D_k^{pol}$) the corresponding manifold $X_t$ contains an algebraically coisotropic subvariety of codimension $k$ which is covered by $\mathbb P^k$'s. Moreover, if $t\in \mathfrak D_k$ (respectively $t\in\mathfrak D_k^{pol}$) is a general point, $X_t$ is not isomorphic   
to a Hilbert scheme of $n$ points on a $K3$ surface or to a generalized Kummer manifold. 
\end{corollary}
In the polarized case Corollary \ref{cor:moduli}  proves therefore 
Voisin's conjecture \cite[Conjecture 0.4]{Voi15} for any fixed $k$ along a divisorial locus in the above moduli spaces. In particular, the map (\ref{eq:BB}) is surjective 
for points in $\mathfrak D_k^{pol}$. Moreover, such coisotropic subvarieties are of the most special form, as they are covered by projective spaces. Few other  results in this direction are available for varieties not isomorphic to a Hilbert scheme of points on a $K3$ surface or to a generalized Kummer manifold: Voisin showed existence for  the LLSV 8-folds \cite[Corollary 4.9]{Voi15}, Lin showed in \cite{Lin} existence for $k=n$ for projective holomorphic symplectic manifold having a Lagrangian fibration, and, as recalled above, in \cite{CP} the existence of uniruled divisors on deformations of $K3^{[n]}$ is showed. 

Notice that by \cite{V95} a generic holomorphic symplectic manifold (so in particular non-projective) can only contain  holomorphic symplectic submanifolds. 
Corollary \ref{cor:moduli}  provides in this case the
existence   Voisin's coisotropic subvarieties of fixed codimension on a locus of the largest possible dimension in the moduli space of holomorphic symplectic manifolds deformation equivalent to one of the two infinite series of examples.

We conclude the introduction with some words on the proof of our main result. 
Following \cite{Voi89} one checks that the Hodge locus of the class of an isotropic submanifold is always smooth (see Section 3). 
On the other hand, using the $T^1$-lifting principle we show that, under the hypotheses of Theorem \ref{thm:stability-intro},  $\Def (F,X)$ is also smooth (cf. Theorem \ref{thm defo fiber}). 
The equality between these two loci can then be checked at the level of tangent spaces, which we do in Theorem \ref{thm:stability}.

\subsection*{Acknowledgments.} 
C.L. was supported by the DFG through the research grant Le 3093/2-1.
G.P. was partially supported by the University of Strasbourg Institute for Advanced Study (USIAS), as part of a USIAS Fellowship.


\section{(Co)isotropic subvarieties}\label{section coiso}
For the basic properties of holomorphic symplectic manifolds we refer the reader to \cite{Huy99} and \cite[Part III]{GHJ}.
We recall some definitions and results, the presentation follows \cite{Voi15}. Let $X$ be a smooth symplectic manifold of dimension $2n$ and denote by $\sigma \in \Gamma(X,\Omega_X^2)$ its symplectic form.

\begin{definition}
A subvariety $F \subset X$ is called \emph{isotropic} if $\sigma\vert_{F^\reg}=0$ where $F^\reg$ denotes the regular part of $F$.
\end{definition}

\begin{definition}\label{def:main}
A subvariety $P \subset X$ is called \emph{coisotropic} if $T_{P^\reg}^\perp \subset T_{P^\reg}$ where $P^\reg$ denotes the regular part of $P$ and $T_{P^\reg}^\perp = \{v\in T_X\vert_{P^\reg} \mid \sigma(v,\cdot)=0\}$ is the radical of the symplectic form. The variety $P$ is called \emph{algebraically coisotropic} if the foliation $T_{P^\reg}^\perp$ is algebraically integrable. This means that there is a rational map $\phi: P \ratl B$ and a holomorphic $2$-form $\sigma_B$ such that over an open dense subset where $\pi$ is defined we have $T_{P/B}=T_P^\perp$. 
We say that $P$ is \emph{smoothly algebraically coisotropic} if it is algebraically coisotropic, the map $\phi$ is almost holomorphic (recall that a rational map is almost holomorphic if it is defined and proper on a dense open subset) and the general fiber of $\phi$ is smooth. The map  $\phi: P \ratl B$ is called in this case a \emph{coisotropic fibration}. 
We say that an isotropic submanifold $F\subset X$ is {\it general fiber of a 
coisotropic fibration} if there exists a smoothly algebraically coisotropic subvariety $P\subset X$ such that $F$ is isomorphic to the general fiber of the associated coisotropic fibration. 
\end{definition}

\begin{remark} ---
\begin{enumerate} 
\item Note that $P$ is automatically smooth in a neighborhood of $F$ by generic flatness and \cite[Exc. III.10.2]{Hart} or its analogue in the analytic category.

\item Let $\tilde \pi :\tilde P \to \tilde B$ be a birational model of $P\ratl B$ with smooth varieties $\tilde P$ and $\tilde B$. One easily shows that there exists a holomorphic $2$-form $\sigma_B$ on $\tilde B$ which is generically non-degenerate and satisfies $\sigma\vert_{\tilde P}=\tilde\pi^*\sigma_B$.
	
\item If an algebraically coisotropic subvariety $P$ has codimension $k$, then from the previous item and the non-degeneracy of the symplectic form one deduces that $B$ has dimension $2n-2k$. 	
	\item Note that an algebraically coisotropic subvariety $P$ which is smooth and such that $\phi:P \ratl B$ coincides with the mrc-fibration of $P$ is smoothly algebraically coisotropic. This follows from \cite{Ca92}.
\end{enumerate}
\end{remark}

If $P \subset X$ is algebraically coisotropic, then the general fiber $F$ of the corresponding map $\phi:P \ratl B$ is isotropic. 
The following lemma is well-known and easily proven by symplectic linear algebra.

\begin{lemma}\label{lemma coisotropic}
 A subvariety $P\subset X$ of codimension $k$ in a symplectic manifold of dimension $2n$ is coisotropic if and only if the symplectic form $\sigma$ on $X$ satisfies $\sigma^{n-k+1}\vert_{P^\reg}=0$.\qed
\end{lemma}

\section{Description of the Hodge locus}\label{section hodge locus}

Let $X$ be a holomorphic symplectic manifold and let $\defx$ be its Kuranishi space. We know by the Bogomolov-Tian-Todorov theorem that $\defx$ is smooth. It is a space germ, but we will always have chosen a representative, which by smoothness  we may assume to be biholomorphic to a small complex ball of dimension $h^ {1,1}(X)$. It is well-known that the universal family $\gothX\to\defx$ is a family of holomorphic symplectic manifolds in a neighborhood of $[X]\in\defx$. For $t\in \defx$ we will denote by $X_t$ the holomorphic symplectic manifold corresponding to $t$.

For an isotropic subvariety $F \subset X$ (not necessarily smooth) we will recall Voisin's description of the Hodge locus $\Hdg_F$ in \cite{Voi89}. It was originally formulated for Lagrangian subvarieties but it carries over literally to the case of isotropic subvarieties. We will denote by $\scrH^k$ the holomorphic vector bundles on $\defx$ whose fiber at $t\in \defx$ is just $H^k(X_t,\C)$. The class $[F] \in H^{2k}(X,\C)$ has a unique flat lifting to $\scrH^{2k}$ which we will also denote by $[F]$.

We denote by $s \in \scrH^2$ a section which fiberwise is the class $[\sigma_t]$ of a symplectic form on $X_t$. Let  $S_{\cup[F]}\subset\defx$ be the subspace of $\defx$ defined by the vanishing of the section $s\cup [F] \in \scrH^{2k+2}$.
Set-theoretically, $S_{\cup[F]}$ can thus be described as $\{t\in S: \sigma_t\cup[F]=0 \in H^{2k+2}(X_t,\C)\}$. As in \cite[Proposition 1.2]{Voi89} one shows

\begin{proposition}\label{proposition voisin1}
Let $F \subset X$ be an isotropic subvariety in a holomorphic symplectic manifold. Then 
$\Hdg_F = S_{\cup[F]}$ and it is a smooth subvariety of $\defx$ 
of codimension 
equal to the rank of cup-product map
$$
 \mu_F:H^2(X,\C)\to H^{2k+2}(X,\C),\ \eta\mapsto \eta\cup[F].
$$
\end{proposition}
\begin{proof}
One can argue exactly as in \cite[Proposition 1.2]{Voi89}, replacing the Lagrangian subvariety, its cohomology class and its dimension $n$, with the isotropic subvariety $F$, its cohomology class and its dimension $k$. 
\end{proof}

We also have the following description.
\begin{prop}\label{proposition voisin2}
With the notation above we have that $\Hdg_F$ coincides with the locus
\begin{equation}\label{eq hodge locus2}
\{t\in S: (j_t)^*\sigma_t=0\}
\end{equation}
where $(j_t)^*$ denotes the composition of the inclusion $j:F\hookrightarrow X$
with the isomorphism $H^2(X_t,\C)\cong H^2(X,\C)$
\end{prop}
\begin{proof}
As in \cite[Proposition 1.7]{Voi89} this follows from Proposition \ref{proposition voisin1} together with 
\cite[Lemme 1.5 and Remarque 1.6]{Voi89}.
\end{proof}

As a consequence, the tangent space to the Hodge locus is given by

\begin{prop}\label{proposition voisin3}
Let $F \subset X$ be an isotropic subvariety in a holomorphic symplectic manifold. Then $$T_{\Hdg_F,[X]}=\{v\in H^1(T_X) \mid j^*\sigma(v,\cdot) = 0 \textrm{ in } H^1(\Omega_F)\}$$
\end{prop}
\begin{proof}
As in \cite[2.2]{Voi89}, this is a direct consequence of \eqref{eq hodge locus2}. Indeed, by differentiating the equation there we obtain the equation $j^*\sigma(v,\cdot) = 0 $ from the properties of the Gau\ss -Manin connection.
\end{proof}

\section{Deformations}\label{sec defo}

A deformation of a compact complex variety $X$ is a flat morphism $\scrX \to S$ of complex spaces to a pointed space $(S,0)$ such that the fiber $\scrX_0$ over $0\in S$ is isomorphic to $X$. 
We will mostly work with space germs but usually take representatives of these germs to work with honest complex spaces and shrink them whenever necessary to small neighborhoods of the central fiber. More precisely, in many situations we choose a deformation over a contractible open subset as a representative of the (uni-)versal deformation of $\scrX\to\Def(X)$.
\subsection{Preparations}\label{sec:prepa}

We start with the following easy remarks.
\begin{lemma}\label{lemma constant cycle}
Let $F$ be a connected, closed 
submanifold of a compact K\"ahler manifold $X$ such that $h^{s,0}(F)=0$ for $1 \leq s \leq p$. Let $\gothF \subset S\times X$ be a deformation of $F$ over a connected complex manifold $S$. If $\pi: \gothF \to S$ is smooth, then $H^0(\gothF,\Omega_\gothF^p)=H^0(S,\Omega_S^p)$ for all $p\in \Z$.
\end{lemma}
\begin{proof}
The exact sequence
$ 0\to \pi^*\Omega_S \to \Omega_\gothF \to \Omega_{\gothF/S} \to 0$
gives rise to a filtration on $\Omega_\gothF^p$ whose graded pieces are the $\pi^*\Omega_S^r \tensor \Omega_{\gothF/S}^s$ such that $r+s=p$. Now $H^0(\gothF,\pi^*\Omega_S^r \tensor \Omega_{\gothF/S}^s)=H^0(S,\Omega_S^r \tensor \pi_*\Omega_{\gothF/S}^s)=0$ for $s\neq 0$ as $\pi_*\Omega^s_{\gothF/S}$ is the zero sheaf. This in turn is deduced from the vanishing of $H^0(F,\Omega^s_F)$ for $0\neq s \leq p$ and the invariance of $h^0(F,\Omega^s_F)$ under deformation.  The last property follows from \cite[Th\'eor\`eme 5.5]{Deligne} when $X$ is an algebraic variety and
when $X$ is a K\"ahler manifold, the argument of \cite[Th\'eor\`eme 5.5]{Deligne} works literally, we only have to replace the reference to \cite[Thm 6.10.5]{EGAIII2} by \cite[Ch~3, Thm 4.1]{BS} and the reference to \cite[7.8.5]{EGAIII2} by \cite[Ch~3, Cor 3.10]{BS}.
The lemma follows now from connectedness of~$F$.
\end{proof}
\begin{lemma}\label{lemma remains isotropic}
Let $F\subset X$ be a submanifold of a symplectic manifold and let $\gothF \into \gothX$ be a deformation of $F\into X$ over a connected base scheme $S$ such that $\gothX \to S$ and $\gothF\to S$ are smooth. If $F$ is isotropic, then $\gothF$ remains isotropic for every relative holomorphic $2$-form on $\gothX$.
\end{lemma}
\begin{proof}
We look at the restriction morphism in de Rham cohomology $\vrho:H^2_\dR(\gothX/S)\to H^2_\dR(\gothF/S)$ where $H^2_\dR(\cdot/S)$ is the second direct image of the complex $\Omega_{\cdot/S}^\bullet$ which is a locally free sheaf on $S$ by Deligne's theorem \cite[Th\'eor\`eme 5.5]{Deligne}. We have to show that the restriction of $\vrho$ to the Hodge filtration $F^2H^2_\dR(\gothX/S)$ vanishes identically.
As in \cite{Deligne} one successively reduces to the case where $S$ is affine respectively Artinian. Then by \cite[Theorem 4.23]{L2} the cokernel is $\sO_S$-free and by \cite[Theorem 4.17]{L2} the same is true for the cokernels of the graded pieces. So $\Gr^2\vrho$ is the zero map as it is so on the central fiber. The claim follows as $\Gr_F^2 = F^2$.
\end{proof}
\begin{lemma}\label{lemma isotropic}
Let $F \subset X$ be an isotropic submanifold of a symplectic manifold $X$ and let $\gothF \into \gothX$ be a deformation of $F\into X$ over a connected base scheme $S$ such that $\gothX \to S$ and $\gothF\to S$ are smooth and such that there is a relative symplectic form $\sigma$ on $\gothX$ extending the one on $X$. 
If $\gothF$ remains isotropic for $\sigma$, then there is a commutative diagram
\[
\xymatrix{
T_{\gothX/S} \ar[r]\ar[d] & N_{\gothF/\gothX}\ar[d]^\vphi \\
\Omega_{\gothX/S} \ar[r] & \Omega_{\gothF/S}\\
}
\]
where the left vertical isomorphism is induced by the symplectic form. In particular, $\vphi$ is surjective. If moreover $P \subset X$ is a coisotropic subvariety such that $F \subset P$ is a general fiber of the coisotropic fibration of $P$, then the restriction of $\vphi$ to $F$ is the composition of $N_{F/X} \to N_{P/X}\vert_F$ and an isomorphism $N_{P/X}\vert_F \isom \Omega_F$.
\end{lemma}
\begin{proof}
This is rather standard and we will only sketch the proof. It is clearly sufficient to verify the existence of the diagram after restriction to $\gothF$. As $\gothF$ is isotropic, $T_{\gothF/S}$ is contained in its orthogonal $T_{\gothF/S}^\perp$ with respect to $\sigma$. Consider the isomorphism $T_{\gothX/S} \to \Omega_{\gothX/S}$ induced by $\sigma$ and let $I_{\gothF}\subset \sO_\gothX$ be the ideal sheaf of $\gothF$. The image of $T_{\gothF/S}^\perp$ in $\Omega_{\gothX/S}$ is identified with 
$I_\gothF/I_{\gothF^2}$ so that it maps to zero in $\Omega_{\gothF/S}$. Hence, $\vphi$ exists and is surjective. To justify the last claim on simply has to observe that the kernel of $N_{F/X} \to N_{P/X}\vert_F$ is $N_{F/P}$ and this is a quotient of $T_F^\perp = T_P\vert_F$ by the coisotropicity of $P$.
\end{proof}

\subsection{Proof of the main results}\label{sec:proofs}

For the basic material on deformation theory we refer the reader to \cite{Ser}.
Let $f:\mathcal X\to \Def(X)$ be the universal deformation of $X$. Consider the relative Hilbert scheme (or rather the Douady space) of $\mathcal X\to \defx$ and let $\mathcal H$ be the union of those irreducible components of it which contain $F$. We endow $\sH$ with the smallest scheme structure that coincides with the Hilbert scheme in an open neighborhood of $[F]$ in $\sH$. As $\pi$ is proper, which may be seen similarly as \cite[Proposition 2.6]{GLR}, the scheme theoretic image $\deff$ of the natural morphism
$\pi :\mathcal H \to \defx$ is a closed subvariety contained in the Hodge locus $\Hdg_F\subset \defx$ of~$F$. 
\begin{thm}\label{thm defo fiber}
Let $X$ be a compact holomorphic symplectic manifold of dimension $2n$ 
and let $F\subset X$ be an isotropic submanifold of dimension $k=1,\ldots,n$ 
 which is a general fiber of some smoothly coisotropic fibration. If
 $(n-k)h^1(\mathcal O_F)=0$, then the following hold.
\begin{enumerate}
 \item\label{enum i} The space $\sH$ is smooth at $[F]$. 
 \item\label{enum ii} The morphism $\sH \to \Def(F,X)$ and  the subspace $\Def(F,X) \subset \Def(X)$  are smooth at $[F]$.
 \item\label{enum iii} If the representative of $\defx$ is chosen small enough, $\sH$ is irreducible and so is the general fiber of $\sH \to \deff$.
\end{enumerate}
\end{thm}
We will henceforth choose the representative of $\defx$ so such that the conclusion of \eqref{enum iii} is fulfilled.
\begin{proof}
Consider the exact sequence
\begin{equation}\label{eq t1}
0\to T_X\langle F\rangle \to T_X \to N_{F/X} \to 0
\end{equation}
of sheaves on $X$ which defines $T_X\langle F\rangle$.
Then $H^1(X,\txf)$ is the tangent space to the relative Hilbert scheme $\sH$ at $[F]$. Let $S=\Spec R$ be a local Artinian scheme of finite type over $\C$ and let $\gothF \into \gothX$ be a deformation of $F \into X$ over $S$. Smoothness of $\sH$, $\deff$ and $\pi$ follows via the $T^1$-lifting principle (cf. \cite[\S 14]{GHJ} for a concise account) if we can show that $H^1(\gothtxf)$ and $\img\left(H^1(\gothtxf)\to H^1(T_{\gothX/T})\right)$ are free $R$-modules. Moreover, $\img\left(H^1(\txf)\to H^1(T_X)\right)$ will be the tangent space to $\deff$ at $[F]$ in this case.

As $H^1(T_X)=H^1(\Omega_X)=0$ we have also $H^1(T_{\gothX/S})=H^1(\Omega_{\gothX/S})=0$. So there is an exact sequence
\begin{equation}\label{eq txf sequence}
 0 \to H^0(N_{\gothF/\gothX}) \to H^1(\gothtxf) \to  H^1(\Omega_{\gothX/S}) \to H^1(N_{\gothF/\gothX}).
\end{equation}
Suppose that $P\ratl B$ is the coisotropic fibration of a smoothly algebraically coisotropic subvariety $P\subset X$ such that $F$ is a general fiber of $P\ratl B$. Then on $X$ we have a short exact sequence
\begin{equation}\label{eq normal}
0\to N_{F/P} \to N_{F/X} \to N_{P/X}\vert_F \to 0
\end{equation}
where the first term is isomorphic to $\sO_F^{2n-2k}$ and the last term to $\Omega_F$ by Lemma \ref{lemma isotropic}. As $\gothF \into \gothX$ is isotropic by Lemma \ref{lemma remains isotropic} there is a relative version of this sequence 
\begin{equation}\label{eq normal2}
  0\to K \to N_{\gothF/\gothX} \to \Omega_{\gothF/S} \to 0
\end{equation}
which gives back \eqref{eq normal} when restricted to the central fiber. Note that all sheaves in \eqref{eq normal2} are $S$-flat.
In order to show that $K$ is also the trivial bundle it suffices to show that the restriction to the central fiber $H^0(K) \to H^0(K\tensor_R \C)=H^0(\sO_F^{2n-2k})$ is surjective. This is easily shown by induction. 
Indeed, suppose that $N\in \N$ is such that $\gothm_R^N=0$ where $\gothm_R$ is the maximal ideal of $R$ and denote $S_m:=\Spec R/\gothm_R^{m+1}$. Then it suffices to inductively show that $K\tensor \sO_{S_m} =\sO_{\gothF\times_S S_m}^{2n-2k}$ for all $m\in \N$.
Fix $m\in \N$ and suppose that $K\tensor \sO_{S_m}$ is a free $\sO_{\gothF\times_S S_m}$-module. Consider the exact sequence 
$$
0\to \sO_F^{2n-2k} \tensor_\C \left(\gothm^{m+1}/\gothm^{m+2}\right) \to K\tensor \sO_{S_{m+1}} \to K\tensor \sO_{S_m}\to 0
$$ 
obtained by flatness of $K$. Then $H^0(K\tensor \sO_{S_{m+1}}) \to H^0(K\tensor \sO_{S_m})$ is surjective by the vanishing of $H^1(\sO_F^{\oplus2(n-k)})$ and the triviality of $K$ follows.

We deduce that also $H^1(K)=H^1(\sO_\gothF^{2n-2k})=0$ and the map $H^1(N_{\gothF/\gothX}) \to H^1(\Omega_{\gothF/S})$ is injective. Combining this with \eqref{eq txf sequence}, we obtain an exact sequence
\begin{equation}\label{eq txf sequence 2}
 0 \to H^0(N_{\gothF/\gothX}) \to H^1(\gothtxf) \to  H^1(\Omega_{\gothX/S}) \to H^1(\Omega_{\gothF/S})
\end{equation}
where the last map is just the ordinary restriction map by Lemma \ref{lemma isotropic}. We observe that the last two terms are $R$-free by \cite[Th\'eor\`eme 5.5]{Deligne} and so is the cokernel, hence the image and the kernel, of the last map in this sequence by \cite[Theorem 4.17]{L2}. Also $H^0(N_{\gothF/\gothX})$ is $R$-free because of the sequence \eqref{eq normal2} and \cite[Th\'eor\`eme 5.5]{Deligne}. 
Consequently, we  deduce the freeness of $H^1(\gothtxf)$ and conclude the proof of \eqref{enum i} and \eqref{enum ii}.

For \eqref{enum iii}, consider the normalization $\tilde\sH \to \sH$. Note that as $\sH$ is smooth in a neighborhood of $[F]\in \sH$, the scheme structure on $\sH$ is reduced and normalization is well-defined. Smoothness of $\sH$ in a neighborhood of $[F]$ also yields a section of $\sH \to \defx$ which induces a section of the composition $\tilde \sH \to \sH \to \defx$ as the normalization is an isomorphism over the smooth locus. This is where one might need to shrink the representative. Let $\tilde\sH \to D \to \defx$ be the Stein factorization of $\tilde \sH \to \defx$. As $D\to\defx$ is finite, admits a section, and $\defx$ is smooth, $D$ has a component isomorphic to $\defx$. But $\tilde\sH$ is integral, so $D$ has also to be integral and $D=\defx$. It follows that the general fiber of $\sH\to\defx$ is connected.
From resolution of singularities together with generic smoothness one deduces that the general fiber of a morphism with connected general fiber from an irreducible space to a smooth variety has irreducible generic fiber, so we infer that this is true for $\sH\to\defx$.
\end{proof}
The fact that $H^1(\Omega_{\gothX/S}) \to H^1(\Omega_{\gothF/S})$ has a locally free cokernel could also be deduced from a Katz-Oda type argument, see \cite{ko}, but this only seems to work when $S$ is an infinitesimal truncation of a smooth variety, a restriction which is not necessary with the above argument. We deduce the following 
\begin{corollary}\label{corollary hilbertscheme}
In the situation of Theorem \ref{thm defo fiber} for every small deformation $F_t\subset X_t$ of $F\subset X$ the Hilbert scheme $\Hilb(X_t)$ is smooth at $[F_t]$ of dimension $2n-2k$.
\end{corollary}
\begin{proof}
 The Hilbert scheme $\Hilb(X_t)$ is locally at $[F_t]$ given by the fiber of $\sH \to \Def(F,X)$ over $t \in \Def(F,X)$ which is smooth at $[F_t]$ by Theorem \ref{thm defo fiber}. Hence, also the fiber dimension of $\sH\to \Def(F,X)$ is constant and equal to $\dim H^0(N_{F/X})=2n-2k$.
\end{proof}

\begin{corollary}\label{cor:covered}
In the situation of Theorem \ref{thm defo fiber} for any $t\in \Def (F,X)$ the deformation $X_t$ of $X$ contains a codimension $k$ algebraically coisotropic subvariety $P_t$ covered by deformations $F_t$ of $F$. 
\end{corollary}
\begin{proof}
Let $\sF \subset \sH\times_{\Def(X)}\sX$ be the universal family of deformations of $F$ parametrized by $\sH$. Denote by $F_t$ its fiber over $t\in \sH$. 
We infer from the proof of Theorem \ref{thm defo fiber} that there is an exact sequence  
$ 0\to \sO_{F_t}^{\oplus (n-k)} \to N_{F_t/X_t} \to \Omega_{F_t} \to 0$ for all $t \in \sH$ sufficiently close to $[F]$.
  Hence, the evaluation morphism $H^0(N_{F_t/X_t}) \tensor \sO_{F_t} \to N_{F_t/X_t}$ is injective with cokernel $\Omega_{F_t}$ for all such $t$. 
This means that small deformations of $F_t$ inside $X_t$ do not intersect $F_t$.
Let $H_t$ be a component of $\sH_t$ parametrizing smooth deformations of $F$ and denote by $\sF_t \to H_t$ the restriction of the universal family.
By the preceding corollary (or the exact sequence above) the evaluation map $ev_t:\sF_t \to X_t$ maps onto a subvariety $P_t\subset X_t$ of dimension $\dim(P_t)=2n-k$ and it is generically finite.
Let $p\in P_t$ be a smooth point which is covered by smooth fibers of $\sF_t \to H_t$.
Analytically locally around $p$ the variety $P_t$ has an isotropic fibration whose fibers are the $F_t$'s, in other words, there is an analytically open subset of $P_t$ which is isomorphic to the universal family $\sF_S \subset X_t \times S$ restricted to a small analytically open subset $S \subset H_t$ so that by Lemma \ref{lemma constant cycle}. But then $\sigma^{n-k+1}\vert_{P_t^\reg}=0$ as the analytically open sets of the form $\sF_S$ cover a Zariski open subset of $P_t$ by definition of $P_t$ and so $P_t$ is coisotropic by Lemma \ref{lemma coisotropic}. In particular, $T_P^\perp$ is the relative tangent bundle of $\sF_t\to H_t$ in a small analytic neighborhood of $p$ so that there is only one $F_t$ passing through $p$ and thus $ev_t:\sF_t \to P_t$ is birational.
We deduce that $P_t$ algebraically coisotropic.
\end{proof}

\begin{remark}\label{rmk:defP}
Note that we do however not know whether $P_t$ is a deformation of $P$. One can only say that for general $t\in \deff$ the variety $P_t$ is a deformation of a variety $P_0$ one of whose components is $P$. Also, we do not know whether $P_t$ is smoothly algebraically coisotropic.
\end{remark}


We prove now the following.

\begin{thm}\label{thm:stability}
In the situation of Theorem \ref{thm defo fiber} we have $\Def (F,X)=\Hdg_F$ and
\begin{equation}\label{eq codimension}
 \codim_\defx\Def(F,X) = \rank\left(H^2(X,\C) \to H^2(F,\C)\right).
\end{equation}
In particular, if $\codim_\defx \Hdg_F \geq \rk \NS(F)$, 
then
$$
\codim_\defx \Def (F,X)=\codim \Hdg_F=\rk \NS(F).
$$ 
\end{thm}
\begin{proof}
Recall that $T_{\Hdg_F,0} = \ker\left(H^1(T_X) \isom H^1(\Omega_X) \to H^1(\Omega_F)\right)$ just like in Voisin's original paper, see Proposition \ref{proposition voisin3} above.
As $\Def(F,X) \subset \Hdg_F$ and both are smooth (the former by Theorem \ref{thm defo fiber}, the latter by Proposition \ref{proposition voisin1}) it suffices to show that their tangent spaces agree. Now $T_{\Def(F,X),0}$ is the kernel of $H^1(T_X) \to H^1(N_{F/X})$. Recall that $N_{F/P}$ is trivial and that $N_{P/X}\vert_F\isom \Omega_F$, by Lemma \ref{lemma isotropic}.  Then the map $H^1(N_{F/X})\to H^1(\Omega_F)$ induced from \eqref{eq normal} is injective as $(n-k)h^1(\sO_F)=0$ and 
the equality $T_{\Def(F,X),0}=T_{\Hdg_F,0}$ follows.

We also deduce that $\dim \deff = h^{1,1}(X) - \rk\left(H^1(\Omega_X) \to H^1(\Omega_F)\right)$ which implies the statement on the codimension as $\dim \defx  = h^{1,1}(X)$. 
As $F$ is isotropic,  the restriction map $H^2(X,\C)\to H^2(F,\C)$ sends the transcendental part of $H^2(X,\C)$ to zero so that $\rk\left(H^1(\Omega_X) \to H^1(\Omega_F)\right)=\rk\left(\NS(X) \to \NS(F)\right)$ and hence $\codim_\defx \Hdg_F \leq \rk \NS(F)$ so that the inequality in the other direction implies equality.
\end{proof}

\begin{remark}\label{rmk:rho=1}
The hypothesis  $\codim \Hdg_F \geq \rk \NS(F)$ is always fulfilled if the Picard number $\rho_F$ of $F$ is one, see \cite[Theorem 2.3]{V95}. We do not know whether there are examples where the hypothesis may fail. 
\end{remark}

We conclude the section by noticing that Theorem \ref{thm:stability-intro} follows immediately
by putting together Theorem \ref{thm:stability}, Corollary \ref{cor:covered} and Remark \ref{rmk:rho=1}. 

Similarly, Corollary \ref{cor:char} follows from Theorem \ref{thm:stability-intro} together with Propositions \ref{proposition voisin1} and \ref{proposition voisin2}. 
Arguing as in \cite[Corollaire 1.4]{Voi89} and using Theorem \ref{thm:stability-intro} we also deduce Corollary \ref{cor:voisin}.    
\section{Examples and final remarks}\label{sec:examples}

\begin{Ex}\label{ex:K3}
Let $S$ be a $K3$ surface containing a smooth rational curve $R$.
For any $1\leq k\leq n$ consider 
$$
 P:=P_k =\{\xi \in S^{[n]}: \lg(\Supp(\xi)\cap R)\geq k\}. 
$$ 
There is a natural dominant rational map $P\dashrightarrow S^{[n-k]}$, which restricts to a surjective morphism
over the open subset $U^{[n-k]}$, where $U:=S\setminus R$. The fiber $F_k$ over a point $\eta \in U^{[n-k]}$ is $\P^k\cong R^{[k]}$.  

As the moduli space of marked holomorphic symplectic manifolds is constructed via the local Torelli theorem by patching together the Kuranishi spaces it is sufficient to prove on the local chart $\Def(S^{[n]})$ the assertions of Corollary \ref{cor:moduli}. 
A small deformation of projective space remains projective space, so, by Theorem \ref{thm:stability-intro}, for every $t\in \Def(F_k,S^{[n]})$ the manifold $X_t$ contains a coisotropic 
submanifold $P_t$ ruled by $\P^k$'s. As $F_k\isom \P^k$ has Picard rank one, Theorem~\ref{thm:stability} immediately implies that $\gothD_k:=\Def(F_k,S^{[n]})$ is a divisor in $\Def(S^{[n]})$. 


To see that for any $k\geq 1$ and for a general point  $t\in \Def(F_k,S^{[n]})$ the manifold $X_t$ is not isomorphic to 
a $(K3)^{[n]}$ 
we 
can argue as follows. Let $\ell_t\subset \P^k\cong F\subset P_t$ be a line inside a fiber of the  coisotropic submanifold $P_t\subset X_t$. If $X_t$ were isomorphic to 
 a $(K3)^{[n]}$, then $\Def(F,S^{[n]})$ would be contained in the codimension 2 locus of $\Def(S^{[n]})$ given by the intersection of the Hodge loci of the exceptional class and of the class dual to that of $\ell_t$ contradicting the fact that the $\gothD_k$ are divisors. Note that these two Hodge loci cannot coincide as one can see from the fact that a curve whose class is dual to the exceptional class is contracted under the map to the symmetric product whereas $\ell_t$ is not. Alternatively, 
this follows from the calculations in Remark \ref{remark k3n}. 
By the local Torelli theorem the polarized deformations $\Def(S^{[n]})^{pol}$ of $S^{[n]}$ are isomorphic to the preimage under the period map of the orthogonal to a class $h \in H^2(S^{[n]},\mathbb Z)$ of positive square. As of course $\Hdg_F$ does not coincide with $\Def(S^{[n]})^{pol}$, it cuts a divisor on it. 

All the conclusions of Corollary \ref{cor:moduli} are then  proved in this case. 
\end{Ex}

\begin{remark}\label{remark k3n}
Notice that the common intersection of any two of the $\gothD_k$ (respectively $\gothD_k^{pol}$) is a codimension $2$ subset $\gothM$ (respectively $\gothM^{pol}$) contained in the locus of Hilbert schemes of points. 
Indeed, it is  clear that  $\gothD_k$ contains the locus of Hilbert schemes of points on K3 surfaces containing a deformation of $R$. This is a hypersurface in the space of K3 surfaces and their Hilbert schemes thus form a codimension $2$ subset of $\Def(S^{[n]})$. To show that the intersection of any two (and hence of all) $\gothD_k$ is not bigger than this
by Corollary \ref{cor:voisin}
it is sufficient to show that, say, for $k\neq k'$, the kernels of the cup-product maps $\ker(\mu_{F_k})$ and $\ker(\mu_{F_{k'}})$ do not coincide. On the other hand as in \cite[Lemme 1.5]{Voi89} (see also \cite[Remarque 1.6]{Voi89}) we have that $\ker(\mu_{F_k}) = \ker (j_k^*)$ (resp. $\ker(\mu_{F_{k'}}) = \ker (j_{k'}^*)$), where $j_k^*$ and $j_{k'}^*$ are the pull-back maps in cohomology associated to the inclusions $j_k:F_k\hookrightarrow X$ and $j_{k'}:F_{k'}\hookrightarrow X$.
As $F_k$ and $F_{k'}$ are isotropic, the transcendental lattice $T \subset H^2(X,\C)$ is contained in both kernels and one may verify $\ker (j_k^*)\neq \ker (j_{k'}^*)$ on the Neron-Severi sublattice (with $\Q$-coefficients) 
by explicit geometric calculations. Let us be a little more precise: $\NS(S^{[n]})=\NS(S)\oplus \Q E$ where $E$ is the exceptional divisor of $S^{[n]}\to S^{(n)}$ and if $\Sigma \subset S$ is an effective divisor we obtain a divisor on $S^{[n]}$ via
\[
 D_\Sigma = \{\xi \in S^{[n]}\mid \xi \cap \Sigma \neq \emptyset\}.
\]
The restriction of $D_\Sigma$ to $F_k$ is uniquely determined by its degree (as $F_k \isom \P^k$) and it is easy to check that $\deg D_\Sigma\vert_{F_k} = (\Sigma.R)_S$, which is independent of $k$. On the other hand, $E\vert_{F_k}$  has degree $2(k-1)$, see \cite[VIII Proposition 5.1]{ACGH}, so that indeed the kernel of the restriction map for $k$ and $k'\neq k$ is not the same.

\end{remark}

\begin{Ex}\label{ex:Kum}
Let $T$ be a 2-dimensional complex torus containing a smooth elliptic curve $E$. Then $T$ is fibered over the elliptic curve $E':=T/E$ and $E$ is the fiber over zero of this fibration. The fiber $E_t$ for $t\in E'$ is a translate of $E$ by an arbitrary preimage of $t$ under $T\to E'$.
For any $1\leq k\leq n$ and any  divisor $D$ of degree $k+1$ on $E_t$ we have
$h^0(\mathcal O_{E_t}(D))=k+1$. In this way we get an immersion 
$$
 \P^k \hookrightarrow E_t^{(k+1)}\subset  T^{[k+1]}
$$
where $\P^k$ is the fiber over some fixed point $a\in E_{(k+1)t}$ of the sum map $E_t^{(k+1)}\to E_{(k+1)t}$.
We consider the relative Hilbert scheme $\sE^{(k+1)}$ (or relative symmetric product) of $k+1$ points on the fibers of $T\to E'$ and the submanifold $\tilde P \subset T^{[n-k]}\times\sE^{(k+1)}$
given by taking the fiber of the sum map $T^{[n-k]}\times\sE^{(k+1)}\to T$ over $0\in T$. Note that $\tilde P$ has dimension $2n-k$ and the projection to the first factor induces a surjective morphism $\tilde\phi: \tilde P \to T^{[n-k]}$ whose fibers are exactly the $\P^k$'s introduced above.

Define $P:=P_k\subset K_n(T)$ to be the closure of the image of the rational map $\tilde P \ratl K_{n}(T)$ obtained by taking the union of subschemes in $T^{[n-k]}$ and $\sE^{(k+1)}$. Observe that the fibration $\tilde \phi$ induces an almost holomorphic map $\phi: P_k \ratl T^{[n-k]}$ which is the coisotropic fibration of the coisotropic subvariety $P_k$.
Arguing as in the previous example, one checks that all the conclusions of Corollary \ref{cor:moduli} hold also in this case. Hereby we use that the degree of the exceptional divisor restricted to the general fibers $F_k \isom \P^k$ of the coisotropic fibration of $P_k$ is $2k$.
\end{Ex}

\begin{Ex}\label{ex:BN}
Other examples, for Hilbert schemes of points on $K3$ surfaces, come from the classical Brill-Noether theory. 
The idea is the following: if $(S,H)$ is a generic primitively polarized $K3$ surface of genus $g$,
and $n$ is such that the Brill-Noether number $\rho(g,1,n)=g-2(g-n+1)$ is positive, then one can consider the locus $P\subset S^{[n]}$ covered by 
the rational curves associated to degree $n$ non-constant morphisms $\varphi :C\to \mathbb P^1$, as the curve $C\in |H|$ and the morphism $\varphi$ vary. The subvariety $P$ is swept by projective spaces of dimension equal to the codimension of $P$. These projective spaces  are the projectivizations $\mathbb P H^0(S,E)$ of the space of global sections of the Lazarsfeld-Mukai rank two vector bundle $E$ associated to the data of the curve $C$ together with a pencil of degree $n$. 
See  \cite[\S 4.1, Example 3)]{Voi15} and \cite[\S 2, p. 10-11]{Voi02} for the details.
The  point we want to make here is that these examples differ from the previous ones, in the sense that they do not necessarily come from a contraction of $S^{[n]}$.   
\end{Ex}

\begin{remark}
It seems relevant to observe that in all the previous examples one can easily count parameters and check that, independently of the codimension of the coisotropic subvariety $P$, a line $\ell$ inside a general isotropic fiber of $P$ moves in a family of the expected dimension $2n-2$. This ensures (cf.  \cite[Proposition 3.1]{CP} and \cite{R}) that the curve $\ell$ deforms along its Hodge locus $\Hdg_{[\ell]}$. Nevertheless, without Theorem \ref{thm:stability-intro} we cannot control the dimension of the locus that the deformations $\ell_t$ of $\ell$ cover
in $X_t$, for $t\in \Hdg_{[\ell]}$. 
\end{remark}

\begin{remark}
Is it true that $F_t$ has trivial $CH_0$ if $F$ has?
Since the Hodge numbers remain constant under deformation and a variety with trivial group of $0$-cycles satisfies $h^{p,0}=0$ for all $p>0$ by the Mumford-Rojtman theorem, see e.g.  \cite[Theorem 10.4]{Vo2}, 
we have that $h^{p,0}(F_t)=0,\ \forall p>0$. 
The Bloch-Beilinson conjecture implies that $F_t$ should be a subvariety of $X_t$ with  trivial $CH_0$. We do not know how to show this fact unconditionally.
\end{remark}

\begin{remark}
In Theorem \ref{thm:stability-intro} the hypothesis that $F$ covers a codimension $k$ coisotropic submanifold is not necessary, as, already when $k=1$, it is sufficient to have $\dim (\Def (F,X))=2n-2$ to obtain the conclusion (see \cite[Proposition 3.1]{CP} and \cite{R}).
However, we have restricted our attention to this setting because of the importance of algebraically coisotropic subvarieties underlined in \cite{Voi15}.
\end{remark}

\end{document}